\documentclass[12pt]{article}

\usepackage{graphicx}
\usepackage{nameref}
\usepackage[shortlabels,inline]{enumitem}
\usepackage{mathtools}
\usepackage{empheq}
\usepackage[shortlabels]{enumitem}
\setlist[enumerate]{nosep}
\usepackage[doc]{optional}
\usepackage[table]{xcolor}
\usepackage[colorlinks=true,
            linkcolor=refkey,
            urlcolor=lblue,
            citecolor=red]{hyperref}
\usepackage{float}
\usepackage{soul}
\usepackage{graphicx}

\usepackage{booktabs}

\definecolor{labelkey}{rgb}{0,0.08,0.45}
\definecolor{refkey}{rgb}{0,0.6,0.0}
\definecolor{Brown}{rgb}{0.45,0.0,0.05}
\definecolor{lime}{rgb}{0.00,0.8,0.0}
\definecolor{lblue}{rgb}{0.5,0.5,0.99}
\definecolor{rowalt}{HTML}{dde7f0}
\definecolor{tabheadfoot}{HTML}{C0C0C0}

 \usepackage{mathpazo}

\colorlet{hlcyan}{cyan!30}

\colorlet{hlred}{red!40}

\usepackage{stmaryrd}
\usepackage{amssymb}
\makeatletter
\def\namedlabel#1#2{\begingroup
   \def\@currentlabel{#2}%
   \label{#1}\endgroup
}
\makeatother

\newcommand{\sepp}{\setlength{\itemsep}{-2pt}}

\newcommand{\menge}[2]{\big\{{#1}~\big |~{#2}\big\}}

\newcommand{\fenv}[1]%
{\ensuremath{\,\overrightarrow{\operatorname{env}}_{#1}}}
\newcommand{\benv}[1]%
{\ensuremath{\,\overleftarrow{\operatorname{env}}_{#1}}}

\newcommand{\scal}[2]{\left\langle{#1},{#2}  \right\rangle}
\newcommand{\bscal}[2]{\big\langle{#1},{#2}\big\rangle}

\newcommand{\RR}{\ensuremath{\mathbb R}}

\newcommand{\NN}{\ensuremath{\mathbb N}}

\newcommand{\ran}{\ensuremath{{\operatorname{ran}}\,}}

\newcommand{\Id}{\ensuremath{\operatorname{Id}}}

{\begin{list}{}{%
\settowidth{\labelwidth}{\textrm{#1~}}%
\setlength{\leftmargin}{\labelwidth+\labelsep}}}
{\end{list}}
\usepackage{amsthm}
\usepackage[capitalize,nameinlink]{cleveref}
\crefname{equation}{}{equations}
\crefname{chapter}{Appendix}{chapters}
\crefname{item}{}{items}
\crefname{enumi}{}{}

\theoremstyle{definition}
\newtheorem{theorem}{Theorem}[section]

\newtheorem{corollary}[theorem]{Corollary}

\newtheorem{proposition}[theorem]{Proposition}

\newtheorem{example}[theorem]{Example}

\newtheorem{remark}[theorem]{Remark}

\DeclarePairedDelimiter{\lVerts}{\lVert}{\rVert}
\newcommand{\norm}[1]{\ensuremath{\lVerts*{#1}}}

\providecommand{\RR}{\mathbb{R}}

\providecommand{\ran}{\operatorname{ran}}

\providecommand{\Id}{\operatorname{{ Id}}}

\providecommand{\NN}{\mathbb{N}}

\providecommand{\ran}{\operatorname{ran}}

\providecommand{\Id}{\operatorname{Id}}

\providecommand{\RR}{\mathbb{R}}
\providecommand{\NN}{\mathbb{N}}

\providecommand{\mA}{\mathcal{A}}
\providecommand{\mH}{\mathcal{H}}

\definecolor{myblue}{rgb}{.8, .8, 1}

\allowdisplaybreaks 

\begin{document}

\title{\textsc{
Projecting onto rectangular matrices with 
prescribed row and column sums
}}

\author{
Heinz H.\ Bauschke\thanks{
Mathematics, University
of British Columbia,
Kelowna, B.C.\ V1V~1V7, Canada. E-mail:
\texttt{heinz.bauschke@ubc.ca}.},~
Shambhavi Singh\thanks{
Mathematics, University
of British Columbia,
Kelowna, B.C.\ V1V~1V7, Canada. E-mail:
\texttt{sambha@student.ubc.ca}.},~ 
and
Xianfu Wang\thanks{
Mathematics, University
of British Columbia,
Kelowna, B.C.\ V1V~1V7, Canada. E-mail:
\texttt{shawn.wang@ubc.ca}.}
}

\date{October 15, 2021}

\maketitle

\vskip 8mm

\begin{abstract}
{ 
In 1990, Romero presented a beautiful formula for the projection 
onto the set of rectangular matrices with prescribed row and column sums. 
Variants of Romero's formula have been rediscovered by 
Khoury and by Glunt, Hayden, and Reams, for bistochastic (square) matrices
in 1998. 
These results have
found various generalizations and applications. 
}

In this paper, we provide a formula for the more general problem of finding the projection onto the set of 
rectangular matrices with prescribed \emph{scaled} row and column sums. 
Our approach is based on computing the Moore-Penrose inverse of a certain linear operator associated with the problem. In fact, our analysis holds even for Hilbert-Schmidt operators 
{  and we do not have to assume consistency}. We also perform  
numerical experiments featuring the new projection operator. 
\end{abstract}

{


\section{Motivation}

A matrix in $\RR^{n\times n}$ is called \emph{bistochastic} if  all entries of it are nonnegative and all its row and column sums equal $1$. 
More generally, a matrix is \emph{generalized bistochastic} 
if the requirement on nonnegativity is dropped. 
The bistochastic matrices form a convex polytope $B$, 
commonly called the \emph{Birkhoff polytope}, 
in $\RR^{n\times n}$, with its extreme points being 
the permutation matrices (a seminal result due to 
Birkhoff and von Neumann). 
A lovely formula provided in 1998 by Khoury \cite{Khoury} 
--- and also by Glunt et al.~\cite{Gluntetal1} --- 
gives 
the projection { of} any matrix onto $G$, the affine subspace of 
generalized bistochastic matrices 
(see \cref{ex:Takouda} below). 
{ 
More generally, nonnegative \emph{rectangular} matrices with prescribed
row and column sums are called \emph{transportation polytopes}. 
If the nonnegativity assumption is dropped, then 
Romero provided already in 1990 an explicit formula 
(see \cref{r:Romero} below) which even predates the square case!
}
On the other hand, 
the projection onto the set of nonnegative matrices $N$ is simple --- just replace every negative entry by $0$. 
No explicit formula is known to project a matrix onto the set 
of bistochastic matrices; however, because $B=G\cap N$,
one may apply algorithms such as Dykstra's algorithm to iteratively approximate the projection onto $B$ by using 
the projection operators $P_G$ and $P_N$ (see, 
e.g., Takouda's \cite{Takouda05} for details). 
{ 
In the case of transportation polytopes, algorithms
which even converge in \emph{finitely many steps}
were provided by Calvillo and Romero \cite{CaRo}.
}

{\it The goal of this paper is to provide explicit projection 
operators in more general settings.} Specifically, we present 
a projection formula for finding the {\it projection onto 
the set of rectangular matrices
with prescribed \emph{scaled} row and column sums}. 
{  Such problems arise, e.g., in \emph{discrete tomography} \cite{wiki:dt}} {  and the study of \emph{transportation polytopes} \cite{CaRo}.}
Our approach uses
the Moore-Penrose inverse of a certain linear operator $\mA$. 
It turns out
that our analysis works even for Hilbert-Schmidt operators
because the range of $\mA$ can be determined and seen to be 
closed. Our main references are \cite{BC2017}, \cite{KadiRing1} 
{  (for Hilbert-Schmidt operators)}, 
and \cite{Groetsch} {  (for the Moore-Penrose inverse). 
We also note that consistency is not required.}

The paper is organized as follows.
After recording a useful result involving the Moore-Penrose 
inverse at the end of this section, 
we prove our main results in \cref{sec:main}. 
These result are then specialized to rectangular matrices 
in \cref{sec:matrices}. 
We then turn to numerical experiments in \cref{sec:numerics}
{ 
where we compare the performance of three popular algorithms: 
Douglas-Rachford, Method of Alternating Projections, and Dykstra.
}

We conclude this introductory section with a result
which we believe to be part of the folklore (although 
we were not able to pinpoint a crisp reference). 
{ 
It is formulated using the \emph{Moore-Penrose inverse} 
of an operator --- for the definition of the 
Moore-Penrose inverse and its basic properties, 
see \cite{Groetsch} (and also \cite[pages 57--59]{BC2017}
 for a crash course). 
The formula presented works even in the case when the problem
is inconsistent and automatically provides a least squares solution.
}

\begin{proposition}
\label{p:210416a}
Let $A\colon X\to Y$ be a continuous linear operator 
with closed range between two real Hilbert spaces. 
Let $b\in Y$, set $\bar{b}:= P_{\ran A}b$,
and set $C := A^{-1}\bar{b}$.
Then
\begin{equation}
(\forall x\in X)\quad P_Cx = x - A^\dagger(Ax-b),
\end{equation}
{ 
where $A^\dagger$ denotes the Moore-Penrose inverse of $A$.
}
\end{proposition}
\begin{proof}
Clearly, $\bar{b}\in \ran A$;
hence, $C\neq\varnothing$.
Let $x\in X$. 
It is well known 
(see, e.g., \cite[Example~29.17(ii)]{BC2017}) 
that 
\begin{equation}
P_{C}x = x - A^\dagger(Ax-\bar{b}).
\end{equation}
On the other hand,
\begin{equation}
A^\dagger\bar{b}
=A^\dagger P_{\ran A}b
=A^\dagger AA^\dagger b%
= A^\dagger b
\end{equation}
using the fact that 
$AA^\dagger = P_{\ran A}$ 
(see, e.g., \cite[Proposition~3.30(ii)]{BC2017})
and $A^\dagger A A^\dagger = A^\dagger$ 
(see, e.g., \cite[Section~II.2]{Groetsch}).
Altogether, $P_Cx=x-A^\dagger(Ax-b)$ as claimed.
\end{proof}

\section{Hilbert-Schmidt operators}

\label{sec:main}

From now on, we assume that 
\begin{equation}
\text{
$X$ and $Y$ are 
two real Hilbert spaces},
\end{equation}
which in turn give rise to the real Hilbert space 
\begin{equation}
\mathcal{H} := 
\menge{T\colon X\to Y}{\text{$T$ is Hilbert-Schmidt}}.
\end{equation} 
{  
Hilbert-Schmidt operators encompass rectangular matrices 
--- even with infinitely many entries as long as these are
square summable --- as well as certain integral operators.}
(We refer the reader to 
\cite[Section~2.6]{KadiRing1} for basic results
on Hilbert-Schmidt operators
{ 
and also recommend \cite[Section~VI.6]{ReedSimon}.})
{ 
Moreover,}
$\mathcal{H}$ (is generated by and) contains
{\it rank-one} operators of the form 
\begin{equation}
(v\otimes u)\colon X \to Y \colon x\mapsto \scal{u}{x}v, 
\end{equation}
where $(v,u)\in Y\times X$, and with adjoint
\begin{equation}
\label{e:210413a}
(v\otimes u)^*\colon Y \to X \colon y\mapsto \scal{v}{y}u
\end{equation}
and 
\begin{equation}
\label{e:210419a}
(v\otimes u)^*=u\otimes v.
\end{equation}
Moreover,
\begin{equation}
\label{e:210419b}
u\otimes u = \|u\|^2P_{\,\RR u}
\quad\text{and}\quad
v\otimes v = \|v\|^2P_{\,\RR v}. 
\end{equation}
For the rest of the paper, we fix 
\begin{equation}
\text{$e\in X$ and $f \in Y$,}
\end{equation}
and set 
\begin{equation}
\label{e:defA}
\mathcal{A}
\colon \mathcal{H}\to  Y\times X
\colon T \mapsto (Te,T^*f).
\end{equation}
{ 
\begin{proposition}
$\mathcal{A}$ is a continuous linear operator
and 
$\displaystyle \|\mA\| = \sqrt{\|e\|^2+\|f\|^2}$. 
\end{proposition}
\begin{proof}
Clearly, $\mA$ is a linear operator.
Moreover, $(\forall T\in\mathcal{H})$ 
$\|\mathcal{A}(T)\|^2 = \|Te\|^2 + \|T^*f\|^2
\leq \|T\|_{\textsf{op}}^2\|e\|^2+
\|T^*\|_{\textsf{op}}^2\|f\|^2
\leq \|T\|^2(\|e\|^2+\|f\|^2)$ 
because the Hilbert-Schmidt norm dominates the operator norm.
It follows that $\mathcal{A}$ is continuous and 
$\|\mathcal{A}\|\leq \sqrt{\|e\|^2+\|f\|^2}$.
On the other hand,
if $T = f\otimes e$,
then 
$\|T\| = \|e\|\|f\|$, 
$\mA(T) = (\|e\|^2f,\|f\|^2e)$ and hence 
$\|\mA(T)\|=\|T\|\sqrt{\|e\|^2+\|f\|^2}$. 
Thus $\|\mA\| \geq \sqrt{\|e\|^2+\|f\|^2}$.
Combining these observations, we obtain 
altogether that 
$\|\mA\| = \sqrt{\|e\|^2+\|f\|^2}$. 
\end{proof}
}

We now prove that $\ran\mA$ is always closed. 

\begin{proposition}\textbf{(range of $\mA$ is closed)}
\label{p:therange}
The following hold:
\begin{enumerate}
\item 
\label{p:therange00}
If $e=0$ and $f=0$,  then 
$\ran\mA = \{0\}\times\{0\}$.
\item
\label{p:therange01}
If $e=0$ and $f\neq 0$, then 
$\ran\mA = \{0\}\times X$.
\item
\label{p:therange10}
If $e\neq 0$ and $f=0$, then 
$\ran\mA = Y\times \{0\}$. 
\item
\label{p:therange11}
If $e\neq 0$ and $f\neq 0$, then 
$\ran\mA = \{(f,-e)\}^\perp$. 
\end{enumerate}
Consequently, $\ran\mA$ is always 
a closed linear subspace of $Y\times X$. 
\end{proposition}
\begin{proof}
\cref{p:therange00}: Clear. 

\cref{p:therange01}: 
Obviously, $\ran\mA\subseteq \{0\}\times X$.
Conversely, let $x\in X$ and set
\begin{equation}
T := \frac{1}{\|f\|^2}f\otimes x.
\end{equation}
Then $Te=T0 = 0$ and 
\begin{equation}
T^*f = \frac{1}{\|f\|^2}(f\otimes x)^*f
= \frac{1}{\|f\|^2}\scal{f}{f}x = x
\end{equation}
and thus 
$(0,x)=(Te,T^*f)=\mA(T)\in\ran\mA$.

\cref{p:therange10}: 
Obviously, $\ran\mA\subseteq Y\times \{0\}$.
Conversely, let $y\in Y$ and set
\begin{equation}
T := \frac{1}{\|e\|^2}y\otimes e.
\end{equation}
Then $T^*f=T^*0 = 0$ and 
\begin{equation}
Te = \frac{1}{\|e\|^2}(y\otimes e)e
= \frac{1}{\|e\|^2}\scal{e}{e}y = y
\end{equation}
and thus 
$(y,0)=(Te,T^*f)=\mA(T)\in\ran\mA$.

\cref{p:therange11}: 
If $(y,x)\in\ran\mA$, say
$(y,x)=\mathcal{A}(T)=(Te,T^*f)$ for some $T\in\mH$, 
then
\begin{align}
\scal{f}{y}
&=\scal{f}{Te}
= \scal{T^*f}{e}
=\scal{x}{e},
\end{align}
i.e., $(y,x)\perp (f,-e)$.
It follows that $\ran\mA\subseteq \{(f,-e)\}^\perp$.

Conversely, let $(y,x)\in\{(f,-e)\}^\perp$, i.e., 
$\scal{e}{x}=\scal{f}{y}$.

{\it Case~1:} $\scal{e}{x}=\scal{f}{y}\neq 0$.\\
Set
\begin{equation}
\zeta := \frac{1}{\scal{x}{e}}=\frac{1}{\scal{y}{f}} 
\quad\text{and}\quad
T := \zeta (y\otimes x)\in\mH.
\end{equation}
Note that 
\begin{equation}
Te = \zeta(y\otimes x)e
= \zeta \scal{x}{e}y = y
\end{equation}
and 
\begin{equation}
T^*f = \zeta (y\otimes  x)^*f
= \zeta\scal{y}{f}x = x;
\end{equation}
therefore, $(y,x)=(Te,T^*f)=\mA(T)\in\ran \mA$.

{\it Case~2:} $\scal{e}{x}=\scal{f}{y}=0$.\\
Pick $\xi$ and $\eta$ in $\RR$ such that 
\begin{equation}
\xi \|f\|^2=1
\quad\text{and}\quad 
\eta\|e\|^2=1,
\end{equation}
and set 
\begin{equation}
T := \xi(f\otimes x) + \eta(y\otimes e)\in\mH.
\end{equation}
Then
\begin{equation}
Te = \xi(f\otimes x)e + \eta(y\otimes e)e
= \xi\scal{x}{e}f + \eta\scal{e}{e}y = 0f + \eta\|e\|^2y = y
\end{equation}
and 
\begin{equation}
T^*f = \xi(f\otimes x)^*f + \eta(y\otimes e)^*f
= \xi\scal{f}{f}x + \eta\scal{y}{f}e = \xi\|f\|^2x + 0e = x.
\end{equation}
Thus 
$(y,x)=(Te,T^*f)=\mA(T)\in\ran\mA$. 
\end{proof}



We now turn to the adjoint of $\mA$: 

\begin{proposition}\textbf{(adjoint of $\mA$)}
\label{p:A^*}
We have 
\begin{equation}
\label{e:A^*}
\mathcal{A}^*\colon {Y}\times {X} \to \mathcal{H}\colon
(y,x)\mapsto y\otimes e + f \otimes x.
\end{equation}
\end{proposition}
\begin{proof}
Let $T\in\mathcal{H}$ and $(y,x)\in Y\times X$.
Let $B$ be any orthonormal basis of $X$.
Then 
\begin{subequations}
\begin{align}
\scal{\mathcal{A}(T)}{(y,x)}
&= \scal{(Te,T^*f)}{(y,x)}\\
&= \scal{Te}{y}+\scal{T^*f}{x}\\
&= \scal{e}{T^*y}+\scal{T^*f}{x}\\
&= \sum_{b\in B}\big(\scal{e}{b}\scal{b}{T^*y}+\scal{T^*f}{b}\scal{b}{x}\big)\\
&= \sum_{b\in B}\scal{Tb}{\scal{e}{b}y}
+\sum_{b\in B}\scal{Tb}{\scal{x}{b}f}\\
&= \sum_{b\in B}\scal{Tb}{(y\otimes e)b}
+\sum_{b\in B}\scal{Tb}{(f\otimes x)b}\\
&= \scal{T}{y\otimes e} + \scal{T}{f\otimes x} \\ 
&= \scal{T}{y\otimes e + f \otimes x}
\end{align}
\end{subequations}
which proves the result. 
\end{proof}

We have all the results together to start tackling 
the Moore-Penrose inverse of $\mA$. 

\begin{theorem}\textbf{(Moore-Penrose inverse of $\mA$ part 1)}
\label{t:Adagger}
Suppose that $e\neq 0$ and $f\neq 0$. 
Let $(y,x)\in Y\times X$. 
Then 
\begin{equation}
\label{e:Adagger}
\mathcal{A}^\dagger(y,x)
= 
\frac{1}{\|e\|^2}\Big( y \otimes e - 
\frac{\scal{f}{{y}}}{\|e\|^2+\|f\|^2} f\otimes e\Big) + 
\frac{1}{\|f\|^2}\Big( f \otimes x - 
\frac{\scal{e}{x}}{\|e\|^2+\|f\|^2} f\otimes e\Big). 
\end{equation}
\end{theorem}
\begin{proof}
Set 
\begin{equation}
\label{e:goodfriday1}
(v,u) := 
\bigg(\frac{1}{\|e\|^2}\Big( y - 
\frac{\scal{y}{f}}{\|e\|^2 + \|f\|^2} f\Big), 
\frac{1}{\|f\|^2}\Big( x - 
\frac{\scal{x}{e}}{\|e\|^2 + \|f\|^2} e\Big)
\bigg). 
\end{equation}
Then 
\begin{subequations}
\label{e:goodfriday2}
\begin{align}
\scal{f}{v}
&= \frac{1}{\|e\|^2}
\Big(\scal{f}{y} - \frac{\scal{y}{f}}{\|e\|^2 + \|f\|^2}
\scal{f}{f}\Big)\\
&=\frac{\scal{f}{y}}{\|e\|^2}
\Big(
1- \frac{\|f\|^2}{\|e\|^2 + \|f\|^2}
\Big)\\
&=\frac{\scal{f}{y}}{\|e\|^2}
\cdot 
\frac{\|e\|^2 + \|f\|^2 - \|f\|^2}{\|e\|^2 + \|f\|^2}\\
&= \frac{\scal{f}{y}}{\|e\|^2 + \|f\|^2}, 
\end{align}
\end{subequations}
and similarly 
\begin{equation} 
\label{e:goodfriday3}
\scal{e}{u} = 
\frac{\scal{e}{x}}{\|e\|^2 + \|f\|^2}.
\end{equation}
Substituting \cref{e:goodfriday2} and \cref{e:goodfriday3}
in \cref{e:goodfriday1} yields 
\begin{equation}
\label{e:goodfriday4}
(v,u) = 
\Big(\frac{1}{\|e\|^2}\big( y - 
\scal{f}{{v}} f\big), 
\frac{1}{\|f\|^2}\big( x - 
\scal{e}{u} e\big)
\Big). 
\end{equation}
Thus 
\begin{equation}
\label{e:goodfriday5}
y = \|e\|^2v + \scal{f}{v}f
\quad\text{and}\quad
x = \|f\|^2u + \scal{e}{u}e.
\end{equation}
Therefore, using \cref{e:A^*}, \cref{e:goodfriday4}, 
\cref{e:210413a}, and \cref{e:A^*} again, 
we obtain
\begin{subequations}
\begin{align}
\mathcal{A}^*\mathcal{A}\mathcal{A}^*(v,u)
&= \mathcal{A}^*\mathcal{A}(v\otimes e+f\otimes u)\\
&= \mathcal{A}^*\big((v\otimes e)e + (f\otimes u) e,
(v\otimes e)^* f+(f\otimes u)^* f\big)\\
&= \mathcal{A}^*\big(\|e\|^2v + \scal{e}{u}f,
\scal{f}{v}e+\|f\|^2 u\big)\\
&= \big(\|e\|^2v + \scal{e}{u}f\big)\otimes e + 
f\otimes\big(\scal{f}{v}e+\|f\|^2 u\big)\\
&= \|e\|^2v\otimes e + \scal{e}{u}f\otimes e
+ \scal{f}{v}f\otimes e+\|f\|^2f\otimes u \\
&= 
\big(\|e\|^2v+\scal{f}{v}f \big)\otimes e
+
f\otimes \big(\|f\|^2u +\scal{e}{u}e\big)\\
&=y \otimes e + f \otimes x\\
&=\mathcal{A}^*(y,x).
\end{align}
\end{subequations}
To sum up, 
we found $\mathcal{A}^*(v,u)\in\ran \mathcal{A}^* 
= (\ker \mathcal{A})^\perp$ 
such that 
$\mathcal{A}^*\mathcal{A}\mathcal{A}^*(v,u) 
= \mathcal{A}^*(y,x)$.
By \cite[Proposition~3.30(i)]{BC2017}, 
\cref{e:goodfriday4}, and \cref{e:A^*}, 
we deduce that 
\begin{subequations}
\begin{align}
\mathcal{A}^\dagger(y,x)
&=\mathcal{A}^*(v,u)\\
&=\mathcal{A}^*\Big(\frac{1}{\|e\|^2}\big( y - 
\scal{f}{{v}} f\big), 
\frac{1}{\|f\|^2}\big( x - 
\scal{e}{u} e\big)
\Big)\\
&=\frac{1}{\|e\|^2}\big( y\otimes e - 
\scal{f}{{v}} f\otimes e\big) + 
\frac{1}{\|f\|^2}\big( f\otimes x - 
\scal{e}{u} f\otimes e\big)
\end{align}
\end{subequations}
which now results in \cref{e:Adagger}
by using \cref{e:goodfriday2} and \cref{e:goodfriday3}. 
\end{proof}

\begin{theorem}{\textbf{(Moore-Penrose inverse of $\mA$ part 2)}}
\label{t:rest}
Let $(y,x)\in Y\times X$. 
Then the following hold:
\begin{enumerate}
\item
\label{t:rest1}
If $e=0$ and $f\neq 0$, then 
$\displaystyle \mathcal{A}^\dagger(y,x)
= \frac{1}{\|f\|^2} f\otimes x$.\\[+2mm] 
\item
\label{t:rest2}
If $e\neq 0$ and $f= 0$, then 
$\displaystyle\mathcal{A}^\dagger(y,x)
= \frac{1}{\|e\|^2}
y\otimes e$.\\[+2mm]
\item 
\label{t:rest3}
If $e=0$ and $f=0$, then
$\mathcal{A}^\dagger(y,x) = 0\in\mathcal{H}$. 
\end{enumerate}
\end{theorem}
\begin{proof}
Let $T\in\mH$. 

\cref{t:rest1}:
In this case, 
$\mathcal{A}(T) = (0,T^*f)$ 
and $\mathcal{A}^*(y,x) = f\otimes x$. 
Let us verify the Penrose conditions \cite[page~48]{Groetsch}.
First, using \cref{e:210413a}, 
\begin{subequations}
\label{e:210418b}
\begin{align}
\mathcal{A}\mathcal{A}^\dagger(y,x)
&= \mathcal{A}\big(\|f\|^{-2}f\otimes x\big)
= \|f\|^{-2}\big((f\otimes x)e,(f\otimes x)^*f\big)\\
&= \|f\|^{-2}\big(0,\scal{f}{f}x\big)
= (0,x)
\end{align}
\end{subequations}
and 
\begin{align}
\bscal{\mathcal{A}\mathcal{A}^\dagger(y,x)}{(v,u)}
= \bscal{(0,x)}{(v,u)} = \scal{x}{u}
= \bscal{\mathcal{A}\mathcal{A}^\dagger(v,u)}{(y,x)}
\end{align}
which shows that $\mathcal{A}\mathcal{A}^\dagger$ is 
indeed { self-adjoint}.

Secondly, 
\begin{equation}
\label{e:210418a}
\mathcal{A}^\dagger\mathcal{A}(T)
= \mathcal{A}^\dagger(Te,T^*f)
= \mathcal{A}^\dagger(0,T^*f)
= \|f\|^{-2}f\otimes(T^*f), 
\end{equation}
and if 
$S\in\mH$ and 
$B$ is any orthonormal basis of $X$, 
then 
\begin{subequations}
\begin{align}
\bscal{\mathcal{A}^\dagger\mathcal{A}(T)}{S}
&= \|f\|^{-2}\bscal{f\otimes (T^*f)}{S}\\
&= \|f\|^{-2}\sum_{b\in B}\bscal{(f\otimes(T^*f))b}{Sb}\\
&= \|f\|^{-2}\sum_{b\in B}\bscal{\scal{T^*f}{b}f}{Sb}\\
&= \|f\|^{-2}\sum_{b\in B}\bscal{\scal{f}{Tb}f}{Sb}\\
&= \|f\|^{-2}\sum_{b\in B}\scal{f}{Tb}\scal{f}{Sb}\\
&= \bscal{\mathcal{A}^\dagger\mathcal{A}(S)}{T}
\end{align}
\end{subequations}
which yields the symmetry of $\mathcal{A}^\dagger\mathcal{A}$. 

Thirdly, using \cref{e:210418a} 
and the assumption that $e=0$, we have 
\begin{subequations}
\begin{align}
\mathcal{A}\mathcal{A}^\dagger\mathcal{A}(T)
&=
\mathcal{A}\big( \|f\|^{-2}f\otimes(T^*f)\big)
=\|f\|^{-2}\big(0,(f\otimes(T^*f))^*f \big)\\
&=\|f\|^{-2}\big(0,\scal{f}{f}T^*f\big)
= (0,T^*f)\\
&=\mathcal{A}(T).
\end{align}
\end{subequations}

And finally, using \cref{e:210418b}, we have 
\begin{equation}
\mathcal{A}^\dagger \mathcal{A}\mathcal{A}^\dagger(y,x)
= \mathcal{A}^\dagger(0,x) = \|f\|^{-2}f\otimes x
= \mathcal{A}^\dagger(y,x).
\end{equation}

\cref{t:rest2}: This can be proved similar to \cref{t:rest1}. 

\cref{t:rest3}: 
In this case, $\mathcal{A}$ is the zero 
operator and hence the Desoer-Whalen conditions (see \cite[page~51]{Groetsch}) 
make it obvious that $\mathcal{A}^\dagger$ is the zero operator as well.
\end{proof}

Let us define the auxiliary function
\begin{equation}
\label{e:auxf}
\delta(\xi) := 
\begin{cases}
\xi, &\text{if $\xi\neq 0$;}\\
1, &\text{if $\xi=0$}
\end{cases}
\end{equation}
which allows us to combine the previous two results into one: 

\begin{corollary}
Let $(y,x)\in Y\times X$. Then 
\begin{align}
\mathcal{A}^\dagger(y,x)
&= 
\frac{1}{\delta(\|e\|^2)}\Big( y \otimes e - 
\frac{\scal{f}{{y}}}{\delta(\|e\|^2+\|f\|^2)} f\otimes e\Big)\\
&\qquad + 
\frac{1}{\delta(\|f\|^2)}\Big( f \otimes x - 
\frac{\scal{e}{x}}{\delta(\|e\|^2+\|f\|^2)} f\otimes e\Big). 
\end{align}
\end{corollary}

We now turn to formulas for $P_{\ran \mathcal{A}}$
and $P_{\ran\mathcal{A^*}}$. 

\begin{corollary}\textbf{(projections onto $\ran\mA$ and $\ran\mA^*$)}
\label{c:ranges}
Let $(y,x)\in{Y}\times{X}$
and let $T\in\mH$.
If $e\neq 0$ and $f\neq 0$, 
then
\begin{equation}
\label{e:PranA}
P_{\ran \mA}(y,x)
= \mA\mA^\dagger(y,x) = 
\Big( 
y-\frac{\scal{f}{y}-\scal{e}{x}}{\|e\|^2 + \|f\|^2}f,
x-\frac{\scal{e}{x}-\scal{f}{y}}{\|e\|^2 + \|f\|^2}e
\Big)
\end{equation}
and 
\begin{equation}
\label{e:PranA^*}
P_{\ran \mA^*}(T) = \mA^\dagger \mA(T) 
= 
\frac{1}{\|e\|^2} (Te) \otimes e
+ \frac{1}{\|f\|^2}f\otimes(T^*f)
- \frac{\scal{f}{{Te}}}{\|e\|^2\|f\|^2} f\otimes e. 
\end{equation}
Furthermore, 
\begin{equation}
P_{\ran \mA}(y,x)
= \mA\mA^\dagger(y,x) =
\begin{cases}
(0,x), &\text{if $e=0$ and $f\neq 0$;}\\[+1mm]
(y,0), &\text{if $e\neq 0$ and $f= 0$;}\\[+1mm]
(0,0), &\text{if $e=0$ and $f= 0$;}
\end{cases}
\end{equation}
and 
\begin{equation}
P_{\ran \mA^*}(T)
= \mA^\dagger\mA(T) =
\begin{cases}
\displaystyle \frac{1}{\|f\|^2}f\otimes(T^*f), &\text{if $e=0$ and $f\neq 0$;}\\[+4mm]
\displaystyle \frac{1}{\|e\|^2}(Te)\otimes e, &\text{if $e\neq 0$ and $f= 0$;}\\[+4mm]
0, &\text{if $e=0$ and $f= 0$.}
\end{cases}
\end{equation}
\end{corollary}
\begin{proof}
Using \cite[Proposition~3.30(ii)]{BC2017} and 
\cref{e:Adagger}, we obtain 
{ 
for $e\neq 0$ and $f\neq 0$
}
\begin{subequations}
\begin{align}
&P_{\ran \mA}(y,x)\\ 
&= \mA\mA^\dagger(y,x)\\
&=\mA\bigg(\frac{1}{\|e\|^2}\Big( y\otimes e - 
\frac{\scal{f}{{y}}}{\|e\|^2+\|f\|^2} f\otimes e\Big) + 
\frac{1}{\|f\|^2}\Big( f \otimes x - 
\frac{\scal{e}{x}}{\|e\|^2+\|f\|^2} f\otimes e\Big)\bigg)\\
&= \bigg(\frac{1}{\|e\|^2}\Big( y\otimes e - 
\frac{\scal{f}{{y}}}{\|e\|^2+\|f\|^2} f\otimes e\Big)e
+ \frac{1}{\|f\|^2}\Big( f \otimes x - 
\frac{\scal{e}{x}}{\|e\|^2+\|f\|^2} f\otimes e\Big)e,\\
&\quad 
\frac{1}{\|e\|^2}\Big( y\otimes e - 
\frac{\scal{f}{{y}}}{\|e\|^2+\|f\|^2} f\otimes e\Big)^*f
+ \frac{1}{\|f\|^2}\Big( f \otimes x - 
\frac{\scal{e}{x}}{\|e\|^2+\|f\|^2} f\otimes e\Big)^*f\bigg) \\
&= \bigg(\frac{1}{\|e\|^2}\Big( \scal{e}{e}y - 
\frac{\scal{f}{{y}}}{\|e\|^2+\|f\|^2} \scal{e}{e}f\Big) 
+ \frac{1}{\|f\|^2}\Big( \scal{x}{e}f - 
\frac{\scal{e}{x}}{\|e\|^2+\|f\|^2} \scal{e}{e}f\Big),\\
&\quad 
\frac{1}{\|e\|^2}\Big( \scal{y}{f} e - 
\frac{\scal{f}{{y}}}{\|e\|^2+\|f\|^2} \scal{f}{f} e\Big)
+ \frac{1}{\|f\|^2}\Big( \scal{f}{f} x - 
\frac{\scal{e}{x}}{\|e\|^2+\|f\|^2} \scal{f}{f} e\Big)\bigg)
\\
&= 
\bigg(y - \frac{\scal{f}{{y}}}{\|e\|^2+\|f\|^2}f
+ \frac{\scal{e}{x}}{\|f\|^2}f - \frac{\scal{e}{x}\|e\|^2}{\|f\|^2
\big(\|e\|^2+\|f\|^2 \big)}f , \\
&\quad x - \frac{\scal{e}{{x}}}{\|e\|^2+\|f\|^2}e
+ \frac{\scal{f}{y}}{\|e\|^2}e - \frac{\scal{f}{y}\|f\|^2}{\|e\|^2
\big(\|e\|^2+\|f\|^2 \big)}e
\bigg)\\
&= 
\bigg(
y + \frac{-\scal{f}{y}\|f\|^2+ \scal{e}{x}\big(\|e\|^2+\|f\|^2\big) 
- \scal{e}{x}\|e\|^2}{\|f\|^2\big(\|e\|^2+\|f\|^2\big)} f,
\\
&\quad x + \frac{-\scal{e}{x}\|e\|^2+ 
\scal{f}{y} \big(\|e\|^2+\|f\|^2\big) 
- \scal{f}{y}\|f\|^2}{\|e\|^2\big(\|e\|^2+\|f\|^2\big)} e 
\bigg)\\
&= \bigg( y - \frac{\scal{f}{y}  - \scal{e}{x}}{\|e\|^2+\|f\|^2} f, 
x - \frac{\scal{e}{x}  - \scal{f}{y}}{\|e\|^2+\|f\|^2} e
\bigg)
\end{align}
\end{subequations}
which verifies \cref{e:PranA}. 

Next, using \cite[Proposition~3.30(v)\&(vi)]{BC2017} and 
\cref{e:Adagger}, we have 
\begin{subequations}
\begin{align}
&P_{\ran \mA^*}(T)
= P_{\ran \mA^\dagger}(T)
= \mA^\dagger \mA(T)
= A^\dagger(Te,T^*f)\\
&= 
\frac{1}{\|e\|^2}\Big( (Te) \otimes e - 
\frac{\scal{f}{{Te}}}{\|e\|^2+\|f\|^2} f\otimes e\Big) + 
\frac{1}{\|f\|^2}\Big( f \otimes (T^*f) - 
\frac{\scal{e}{T^*f}}{\|e\|^2+\|f\|^2} f\otimes e\Big) \\
&= 
\frac{1}{\|e\|^2} (Te) \otimes e
+ \frac{1}{\|f\|^2}f\otimes(T^*f) 
- \frac{\scal{f}{{Te}}}{\|e\|^2+\|f\|^2}
\Big(\frac{1}{\|e\|^2} + \frac{1}{\|f\|^2} \Big) f\otimes e\\
&= 
\frac{1}{\|e\|^2} (Te) \otimes e
+ \frac{1}{\|f\|^2}f\otimes(T^*f)
- \frac{\scal{f}{{Te}}}{\|e\|^2\|f\|^2} f\otimes e
\end{align}
\end{subequations}
which establishes \cref{e:PranA^*}. 

If $e=0$ and $f\neq 0$,
then 
\begin{equation}
\mA\mA^\dagger(y,x) = \mA(\|f\|^{-2}f\otimes x)
= \|f\|^{-2}(0,(f\otimes x)^*f)
= \|f\|^{-2}(0,\scal{f}{f}x)=(0,x)
\end{equation}
and 
\begin{equation}
\mA^\dagger\mA(T) = 
\mA^\dagger(0,T^*f) = \frac{1}{\|f\|^2}f\otimes(T^*f).
\end{equation}

The case when $e\neq 0$ and $f=0$ is treated 
similarly. 

Finally, if $e=0$ and $f=0$, then $\mA^\dagger=0$ and 
the result follows.
\end{proof}

\begin{theorem}\textbf{(main projection theorem)}
\label{t:210416b} 
Let $(s,r)\in Y\times X$ and set 
$(\bar{s},\bar{r})=P_{\ran\mA}(s,r)$. 
Then 
\begin{equation}
C := \mA^{-1}(\bar{s},\bar{r}) 
= \menge{T\in\mathcal{H}}{Te = \bar{s} \text{\;and\;}
T^* f = \bar{r}}\neq\varnothing. 
\end{equation}
Let $T\in\mH$. 
If $e\neq 0$ and $f\neq 0$, then 
\begin{subequations}
\label{e:P_C}
\begin{align}
P_C(T) &= 
T  - \frac{1}{\|e\|^2}
\Big( 
(Te-s)\otimes e
-\frac{\scal{f}{Te-s}}{\|e\|^2 + \|f\|^2}f\otimes e
\Big)\\
&\qquad
- \frac{1}{\|f\|^2}
\Big( 
f\otimes (T^*f-r)
-\frac{\scal{e}{T^*f-r}}{\|e\|^2 + \|f\|^2}
f\otimes e  \Big).
\end{align}
\end{subequations}
Moreover, 
\begin{equation}
P_C(T) = 
\begin{cases}
T-\displaystyle \frac{1}{\|f\|^2}f\otimes(T^*f-r), &\text{if $e=0$ and $f\neq 0$;}\\[+3mm]
T-\displaystyle \frac{1}{\|e\|^2}(Te-s)\otimes e, &\text{if $e\neq 0$ and $f=0$;}\\[+3mm]
T, &\text{if $e=0$ and $f=0$.}
\end{cases}
\end{equation}
\end{theorem}
\begin{proof}
Clearly, 
$C\neq\varnothing$.
Now \cref{p:210416a} and \cref{e:defA} 
yield
\begin{align}
P_C(T)
&= 
T-\mA^\dagger(\mA T-(s,r))
= 
T - \mA^\dagger(Te-s,T^*f -r). 
\end{align}

Now we consider 
{ 
all possible
}
cases.
If $e\neq 0$ and $f\neq 0$, then,
using \cref{e:Adagger}, 
\begin{align*}
P_C(T)
&= T-
\frac{1}{\|e\|^2}\Big( (Te-s) \otimes e - 
\frac{\scal{f}{{Te-s}}}{\|e\|^2+\|f\|^2} f\otimes e\Big) \\
&\qquad - 
\frac{1}{\|f\|^2}\Big( f \otimes (T^*f-r)
-  \frac{\scal{e}{T^*f-r}}{\|e\|^2+\|f\|^2} f\otimes e\Big) 
\end{align*}
as claimed. 

Next, 
if $e=0$ and $f\neq 0$, then using
\cref{t:rest}\cref{t:rest1} yields
\begin{align*}
P_C(T)
&= T- \frac{1}{\|f\|^2} f \otimes (T^*f-r). 
\end{align*}

Similarly, 
if $e\neq 0$ and $f= 0$, then using
\cref{t:rest}\cref{t:rest2} yields
\begin{align*}
P_C(T)
&= T- \frac{1}{\|e\|^2} (Te-s) \otimes e. 
\end{align*}
And finally, if $e=0$ and $f=0$, then 
$\mA^\dagger =0$ and hence  
$P_C(T) = T$.
\end{proof}

\begin{remark}
Consider \cref{t:210416b} and its notation.
{ 
If $(s,r)\in\ran\mA$, 
then 
$(\bar{s},\bar{r})=(s,r)$ and 
hence $C = \mA^{-1}(s,r)$ which covers
also the consistent case.
Note that 
}
the auxiliary function defined in \cref{e:auxf}
allows us to combine all four cases into 
\begin{subequations}
\label{e:P_Call4}
\begin{align}
P_C(T) &= 
T  - \frac{1}{\delta(\|e\|^2)}
\Big( 
(Te-s)\otimes e
-\frac{\scal{f}{Te-s}}{\delta(\|e\|^2 + \|f\|^2)}f\otimes e
\Big)\\
&\qquad
- \frac{1}{\delta(\|f\|^2)}
\Big( 
f\otimes (T^*f-r)
-\frac{\scal{e}{T^*f-r}}{\delta(\|e\|^2 + \|f\|^2)}
f\otimes e  \Big).
\end{align}
\end{subequations}
\end{remark}

The last two results in this section 
are inspired by \cite[Theorem~2.1]{Gluntetal1}
and \cite[Theorem on page~566]{Khoury}, respectively.
See also \cref{c:GHR} and \cref{ex:Takouda} below. 

\begin{corollary}
\label{c:thinkbig}
Suppose that $Y=X$, let $e\in X\smallsetminus\{0\}$, 
let $f\in X\smallsetminus\{0\}$, 
set 
\begin{equation}
\label{e:thinkbig1}
E:= \frac{1}{\|e\|^{2}}e\otimes e = P_{\,\RR e}
\quad\text{and}\quad
F:= \frac{1}{\|f\|^{2}}f\otimes f = P_{\,\RR f}, 
\end{equation}
and let $\gamma\in\RR$. 
Then 
\begin{equation}
C := \menge{T\in\mH}{Te=\gamma e\;\text{and}\;T^* f=\gamma f}
\neq\varnothing
\end{equation}
and 
\begin{equation}
(\forall T\in\mH)\quad P_C(T) = 
\gamma\Id+(\Id-F)(T-\gamma\Id)(\Id-E). 
\end{equation}
\end{corollary}
\begin{proof}
The projection identities in \cref{e:thinkbig1}
follow from \cref{e:210419b}.
Note that $\gamma\Id \in C$ 
and hence $C\neq\varnothing$.
We may and do assume without loss of generality 
that $\|e\|=1=\|f\|$.

Now let $T\in\mH$. 
Applying \cref{e:P_C} 
with $r:=\gamma f$ and $s:=\gamma e$, 
we deduce that 
\begin{subequations}
\begin{align}
P_C(T)
&= 
T  - 
\Big( 
(Te-\gamma e)\otimes e
-\frac{\scal{f}{Te-\gamma e}}{2}f\otimes e
\Big)\\
&\qquad 
-\Big( 
f\otimes (T^*f-\gamma f)
-\frac{\scal{e}{T^*f-\gamma f}}{2}
f\otimes e  \Big)\\
&=
T-(Te)\otimes e + \gamma e\otimes e
+\frac{\scal{f}{Te}-\gamma\scal{f}{e}}{2}f\otimes e\\
&\qquad 
- f\otimes (T^*f) + \gamma f\otimes f 
+ \frac{\scal{Te}{f}-\gamma\scal{e}{f}}{2}f\otimes e\\
&=
T - (Te)\otimes e - f\otimes (T^*f)
+\gamma(E+F)
+\big(\scal{f}{Te}-\gamma\scal{e}{f}\big)f\otimes e\\
&= T - TE - FT + \gamma(E+F)
+\big(\scal{f}{Te}-\gamma\scal{e}{f}\big)f\otimes e\\
&=T-TE-FT+\gamma(E+F) +FTE -\gamma FE\\
&=\gamma\Id +T-TE-FT+FTE-\gamma\Id+\gamma E+\gamma F-\gamma FE\\
&=\gamma\Id + (\Id-F)T(\Id-E)-\gamma(\Id-F)(\Id-E)\\
&= \gamma\Id + (\Id-F)(T-\gamma\Id)(\Id-E)
\end{align}
as claimed.
\end{subequations}
\end{proof}

We conclude this section with a beautiful
projection formula that arises when
the last result is specialized even further. 

\begin{corollary}
\label{c:abstractTakouda}
Suppose that $Y=X$, let $e\in X\smallsetminus\{0\}$, 
and set 
\begin{equation}
E:= \frac{1}{\|e\|^{2}}e\otimes e = P_{\,\RR e}.
\end{equation}
Then 
\begin{equation}
C := \menge{T\in\mH}{Te=e=T^* e}
\neq\varnothing
\end{equation}
and 
\begin{equation}
(\forall T\in\mH)\quad P_C(T) = E+(\Id-E)T(\Id-E). 
\end{equation}
\end{corollary}
\begin{proof}
Let $T\in\mathcal{H}$. 
Applying \cref{c:thinkbig} 
with $f=e$ and $\gamma=1$, we obtain 
\begin{subequations}
\begin{align}
P_C(T) 
&= \Id + (\Id-E)(T-\Id)(\Id-E)\\
&= \Id+(\Id-E)T(\Id-E)-(\Id-E)^2\\
&= \Id+(\Id-E)T(\Id-E)-(\Id-E)\\
&= (\Id-E)T(\Id-E)+E
\end{align}
\end{subequations}
because $\Id-E=P_{\{e\}^\perp}$ is idempotent. 
\end{proof}

\section{Rectangular matrices}
\label{sec:matrices}

In this section, we specialize the results of 
\cref{sec:main} to 
\begin{equation}
\text{$X=\RR^n$ and $Y=\RR^m$,}
\end{equation}
which gives rise to 
\begin{equation}
\mH = \RR^{m\times n},
\end{equation}
the space of real $m\times n$ matrices. 
Given $u$ and $x$ in $\RR^n$, and $v$ and $y$ in $\RR^m$, we have 
$v\otimes u = vu^\intercal$, 
$(v\otimes u)x=vu^\intercal x = (u^\intercal x)v$, 
and $(v\otimes u)^*y = (v^\intercal y) u$. 
Corresponding to \cref{e:defA}, we have 
\begin{equation}
\mA\colon \RR^{m\times n}\to \RR^{m+n}\colon T\mapsto
\begin{bmatrix}
Te\\
T^\intercal f
\end{bmatrix}.
\end{equation}
The counterpart of \cref{e:A^*} reads
\begin{equation}
\mathcal{A}^*\colon \RR^{m+n}\to\RR^{m\times n}
\colon 
\begin{bmatrix}
y\\
x
\end{bmatrix}
\mapsto 
ye^\intercal + fx^\intercal.
\end{equation}

Translated to the matrix setting,
\cref{t:Adagger} and \cref{t:rest} turn into:
\begin{theorem}
Let $x\in\RR^n$ and $y\in\RR^m$. 
If $e\neq 0$ and $f\neq 0$, then 
\begin{equation}
\mathcal{A}^\dagger
\begin{bmatrix} y \\ x \end{bmatrix}
= 
\frac{1}{\|e\|^2}\Big( ye^\intercal - 
\frac{{f}^\intercal{{y}}}{\|e\|^2+\|f\|^2} f e^\intercal\Big) + 
\frac{1}{\|f\|^2}\Big( f x^\intercal - 
\frac{{e}^\intercal{x}}{\|e\|^2+\|f\|^2} f e^\intercal\Big). 
\end{equation}
Furthermore, 
\begin{equation}
\mathcal{A}^\dagger
\begin{bmatrix} y \\ x \end{bmatrix}
=
\begin{cases}
\displaystyle \frac{1}{\|f\|^2}fx^\intercal,
&\text{if $e=0$ and $f\neq 0$;}\\[+4mm]
\displaystyle \frac{1}{\|e\|^2}ye^\intercal,
&\text{if $e\neq 0$ and $f=0$;}\\[+4mm]
0,
&\text{if $e=0$ and $f=0$.}
\end{cases}
\end{equation}
\end{theorem}

In turn, \cref{c:ranges} now states the following:
\begin{corollary}
Let $x\in\RR^n$, let $y\in\RR^m$, and let $T\in\RR^{m\times n}$.
If $e\neq 0$ and $f\neq 0$, then 
\begin{equation}
P_{\ran \mA}\begin{bmatrix}y\\x\end{bmatrix}
= \begin{bmatrix}
y\\ x
\end{bmatrix}
-\frac{f^\intercal y - e^\intercal x}{\|e\|^2+\|f\|^2}
\begin{bmatrix}
f\\ -e
\end{bmatrix}
\end{equation}
and 
\begin{equation}
P_{\ran \mA^*}(T)
= \frac{1}{\|e\|^2}Tee^\intercal 
+ \frac{1}{\|f\|^2}ff^\intercal T
-\frac{f^\intercal Te}{\|e\|^2\|f\|^2}fe^\intercal.
\end{equation}
Furthermore, 
\begin{equation}
P_{\ran \mA}\begin{bmatrix}y \\ x \end{bmatrix} =
\begin{cases}
\begin{bmatrix}0\\[-2mm]x\end{bmatrix}, &\text{if $e=0$ and $f\neq 0$;}\\[+4mm]
\begin{bmatrix}y\\[-1mm]0\end{bmatrix},  &\text{if $e\neq 0$ and $f= 0$;}\\[+4mm] 
\begin{bmatrix}0\\[-2mm]0\end{bmatrix}, &\text{if $e=0$ and $f= 0$;}
\end{cases}
\end{equation}
and 
\begin{equation}
P_{\ran \mA^*}(T) = 
\begin{cases}
\displaystyle \frac{1}{\|f\|^2}ff^\intercal T, &\text{if $e=0$ and $f\neq 0$;}\\[+4mm]
\displaystyle \frac{1}{\|e\|^2}Tee^\intercal, &\text{if $e\neq 0$ and $f= 0$;}\\[+4mm]
0, &\text{if $e=0$ and $f= 0$.}
\end{cases}
\end{equation}
\end{corollary}

Next, \cref{t:210416b} turns into the following result:

\begin{theorem}
\label{t:210418c}
Let $r\in\RR^n$, let $s\in\RR^m$, and set 
$[\bar{s},\bar{r}]^\intercal = P_{\ran\mA}[s,r]^\intercal$. 
Then
\begin{equation}
C := \menge{T\in\RR^{m\times n}}{Te=\bar{s}\;\text{and}\;T^\intercal f = \bar{r}}\neq\varnothing.
\end{equation}
Now let $T\in\RR^{m\times n}$.
If $e\neq 0$ and $f\neq 0$, then 
\begin{subequations}
\begin{align}
P_C(T) &= 
T  - \frac{1}{\|e\|^2}
\Big( 
(Te-s)e^\intercal
-\frac{{f^\intercal}(Te-s)}{\|e\|^2 + \|f\|^2}fe^\intercal
\Big)\\[+2mm]
&\qquad
- \frac{1}{\|f\|^2}
\Big( 
f(f^\intercal T-r^\intercal)
-\frac{{e^\intercal}{(T^\intercal f-r)}}{\|e\|^2 + \|f\|^2}
fe^\intercal  \Big).
\end{align}
\end{subequations}
Moreover, 
\begin{equation}
P_C(T) = 
\begin{cases}
T-\displaystyle \frac{1}{\|f\|^2}f(f^\intercal T-r^\intercal), &\text{if $e=0$ and $f\neq 0$;}\\[+3mm]
T-\displaystyle \frac{1}{\|e\|^2}(Te-s) e^\intercal, &\text{if $e\neq 0$ and $f=0$;}\\[+3mm]
T, &\text{if $e=0$ and $f=0$.}
\end{cases}
\end{equation}
\end{theorem}

Let us specialize \cref{t:210418c} further
to the following interesting case:

\begin{corollary}\label{c:rectrowsum}\textbf{(projection onto matrices 
with prescribed row/column sums)}\\
\label{c:210914a}
Suppose that $e=[1,1,\ldots,1]^\intercal\in\RR^n$ and that
$f=[1,1,\ldots,1]^\intercal\in\RR^m$. 
Let $r\in\RR^n$, let $s\in\RR^m$, and set 
$[\bar{s},\bar{r}]^\intercal = P_{\ran\mA}[s,r]^\intercal$. 
Then
\begin{equation}
C := \menge{T\in\RR^{m\times n}}{Te=\bar{s}\;\text{and}\;T^\intercal f = \bar{r}}\neq\varnothing, 
\end{equation}
and for every $T\in\RR^{m\times n}$, 
\begin{subequations}
\label{e:210914b}
\begin{align}
P_C(T) &= 
T  - \frac{1}{n}
\Big( 
(Te-s)e^\intercal
-\frac{{f^\intercal}(Te-s)}{n+m}fe^\intercal
\Big)\\[+2mm]
&\qquad
- \frac{1}{m}
\Big( 
f(f^\intercal T-r^\intercal)
-\frac{{e^\intercal}{(T^\intercal f-r)}}{n + m}
fe^\intercal  \Big).
\end{align}
\end{subequations}
\end{corollary}

{ 
\begin{remark}\textbf{(Romero; 1990)}
\label{r:Romero}
Consider \cref{c:210914a} and its notation.
Assume that $[s,r]^\intercal \in\ran \mA$, 
which is equivalent to requiring that 
$\scal{e}{r}=\scal{f}{s}$ 
(which is sometimes jokingly called the ``Fundamental 
Theorem of Accounting''). 
Then one verifies that the entries of the matrix in 
\cref{e:210914b} are given also expressed by 
\begin{equation}
\label{e:210914c}
\big(P_{C}(T)\big)_{i,j} = T_{i,j}
+ \frac{s_i-(Te)_i}{n}
+ \frac{r_j-(T^\intercal f)_j}{m} + \frac{{f^\intercal}{Te}-{e^\intercal}{r}}{mn}
\end{equation}
for every $i\in\{1,\ldots,m\}$ and $j\in\{1,\ldots,n\}$. 
Formula~\cref{e:210914c} was proved by Romero 
(see \cite[Corollary~2.1]{Romero}) who proved this result
using Lagrange multipliers and who has even a $K$-dimensional extension 
(where \cref{e:210914c} corresponds to $K=2$). 
We also refer the reader to \cite{CaRo} for using 
\cref{e:210914c} to compute the projection onto the transportation polytope.
\end{remark}
}

Next, \cref{c:thinkbig} turns into the following result:

\begin{corollary}\textbf{(Glunt-Hayden-Reams; 1998)} 
\cite[Theorem~2.1]{Gluntetal1} \label{c:GHR}
Suppose that $e$ and $f$ lie in $\RR^{n}\smallsetminus\{0\}$, 
set 
\begin{equation}
E:= \frac{1}{\|e\|^{2}}ee^\intercal 
\quad\text{and}\quad
F:= \frac{1}{\|f\|^{2}}ff^\intercal, 
\end{equation}
and let $\gamma\in\RR$. 
Then 
\begin{equation}
C := \menge{T\in\RR^{n\times n}}{Te=\gamma e\;\text{and}\;T^\intercal f=\gamma f}
\neq\varnothing
\end{equation}
and 
\begin{equation}
(\forall T\in\mH)\quad P_C(T) = 
\gamma\Id+(\Id-F)(T-\gamma\Id)(\Id-E). 
\end{equation}
\end{corollary}

We conclude this section with a particularization of \cref{c:abstractTakouda} which immediately follows
when $X=Y=\RR^n$ and thus $\mH = \RR^{n\times n}$: 

\begin{corollary}
\label{c:concreteTakouda}
Suppose that $e\in \RR^n\smallsetminus\{0\}$, 
and set 
\begin{equation}
E:= \frac{1}{\|e\|^{2}}ee^\intercal. 
\end{equation}
Then 
\begin{equation}
C := \menge{T\in\RR^{n\times n}}{Te=e=T^\intercal e}
\neq\varnothing
\end{equation}
and 
\begin{equation}
(\forall T\in\RR^{n\times n})\quad P_C(T) = E+(\Id-E)T(\Id-E). 
\end{equation}
\end{corollary}

\begin{example}\textbf{(projection formula for 
generalized bistochastic matrices; 1998)}\\ 
(See \cite[Theorem on page~566]{Khoury} and 
\cite[Corollary~2.1]{Gluntetal1}.)
\label{ex:Takouda}
Set 
\begin{equation}
u:=[1,1,\ldots,1]^\intercal\in \RR^n,\;\;
C := \menge{T\in\RR^{n\times n}}{Tu=u=T^\intercal u}, 
\;\;\text{and}\;\;
J:=(1/n)uu^\intercal. 
\end{equation}
Then 
\begin{equation}
(\forall T\in\RR^{n\times n})\quad 
P_C(T)=J+(\Id-J)T(\Id-J). 
\end{equation}
\end{example}
\begin{proof}
Apply \cref{c:concreteTakouda} with $e=u$
for which  $\|e\|^2=n$. 
\end{proof}

\begin{remark} 
For some applications of \cref{ex:Takouda}, we 
refer the reader to \cite{Takouda05} and also
to the recent preprint 
\cite{ACT}.
\end{remark}

{ 
\begin{remark} 
A reviewer pointed out that projection algorithms
can also be employed to solve linear programming problems 
provided a strict complementary condition holds 
(see Nurminski's work \cite{Nurminski}). 
This does suggest a possibly interesting future project: explore 
whether the projections in this paper are useful in 
solving some linear programming problems on rectangular matrices
with prescribed row and column sums. 
\end{remark}
}

\section{Numerical experiments}

\label{sec:numerics}

We consider the problem of finding a rectangular matrix with prescribed row and column sums as well as some 
additional constraints on the entries of the matrix. 
To be specific and inspired by \cite{stack}, 
we seek a real matrix of size $m\times n = 4\times 5$ such that its 
row and column sums are equal to 
 $\bar s:=\begin{bmatrix}32,43,33,23\end{bmatrix}^\intercal$ and 
$\bar r:=\begin{bmatrix}24,18,37,27,25\end{bmatrix}^\intercal$, 
respectively. One solution featuring actually nonnegative integers to
this problem is given by
  \[
  \begin{array}{rrrrr|r}
    9&4&8&4&7&32 \\
    7&9&15&7&5&43 \\
    3&2&9&10&9&33 \\
    5&3&5&6&4&23 \\\hline
    24&18&37&27&25&131 
  \end{array}
  \]
Adopting the notation of \cref{c:rectrowsum}, we see that the set 
\begin{equation}
B := \menge{T\in\RR^{4\times 5}}{Te =\bar{s}\;\text{and}\; T^\intercal f = \bar{r}}\neq\varnothing
\end{equation}
is an affine subspace of $\RR^{4\times 5}$ and that 
an explicit formula for $P_B$ is available through \cref{c:rectrowsum}.
Next, we define the closed convex ``hyper box''
\begin{equation}
\label{e:Aold}
A:=\bigtimes_{\substack{i\in\{1,2,3,4\}\\j\in\{1,2,3,4,5\}}}\big [0,\min\{{\bar s}_i,{\bar r}_j\}\big]. 
\end{equation} 
For instance, the $(1,3)$-entry of any nonnegative integer solution must lie between $0$ and $32=\min\{32,37\}$; thus 
$A_{1,3} = [0,32]$.
The projection of a real number $\xi$ onto the interval 
$[0,\min({\bar s}_i,{\bar r}_j)]$ is given by 
$\max\{0,\min\{{\bar s}_i,{\bar r}_j,\xi\}\}$. 
Because $A$ is the Cartesian product of such intervals, 
the projection operator $P_A$ is nothing but the corresponding
product of interval projection operators. 

Our problem is thus to 
\begin{equation}
\label{e:lastone}
\text{find a matrix $T$ in $A\cap B$.}
\end{equation}

We shall tackle \cref{e:lastone} with three well known algorithms:
Douglas-Rachford (DR), Method of Alternating Projections (MAP), 
and Dykstra (Dyk). Here is a quick review of how these methods 
operate, for a given starting matrix 
$T_0\in\RR^{4\times 5}$ and a current matrix 
$T_k\in\RR^{4\times 5}$. 

DR updates via
\begin{equation}
\tag{DR}
T_{k+1} := T_{k} - P_A(T_k) + P_B(2P_A(T_k)-T_k),
\end{equation}
MAP updates via 
\begin{equation}
\tag{MAP}
T_{k+1} := P_B(P_A(T_k)), 
\end{equation}
and finally 
Dyk initializes also $R_0=0\in\RR^{4\times 5}$
and updates via
\begin{equation}
\tag{Dyk}
A_{k+1} := P_A(T_k+R_k), \;\;
R_{k+1} := T_k+R_k-A_{k+1},\;\; T_{k+1} := P_B(A_{k+1}).
\end{equation}
For all three algorithms, it is 
known that 
\begin{equation}
P_A(T_k) \to \text{some matrix in $A\cap B$;}
\end{equation}
in fact, Dyk satisfies even $P_A(T_k)\to P_{A\cap B}(T_0)$
{ 
 (see, e.g., \cite[Corollary~28.3, Corollary~5.26, and Theorem~30.7]{BC2017}).
}
Consequently, 
for each of the three algorithms, we will focus on the sequence 
\begin{equation}
(P_A(T_k))_{k\in\NN}
\end{equation}
which obviously lies in $A$ and which thus 
prompts the simple feasibility criterion given by 
\begin{equation}
\delta_k := \|P_A(T_k)-P_B(P_A(T_k))\|. 
\end{equation}

\subsection{The convex case}

\label{ss:NN}

\begin{figure}[H]
\centering
\includegraphics[width=0.7\textwidth]{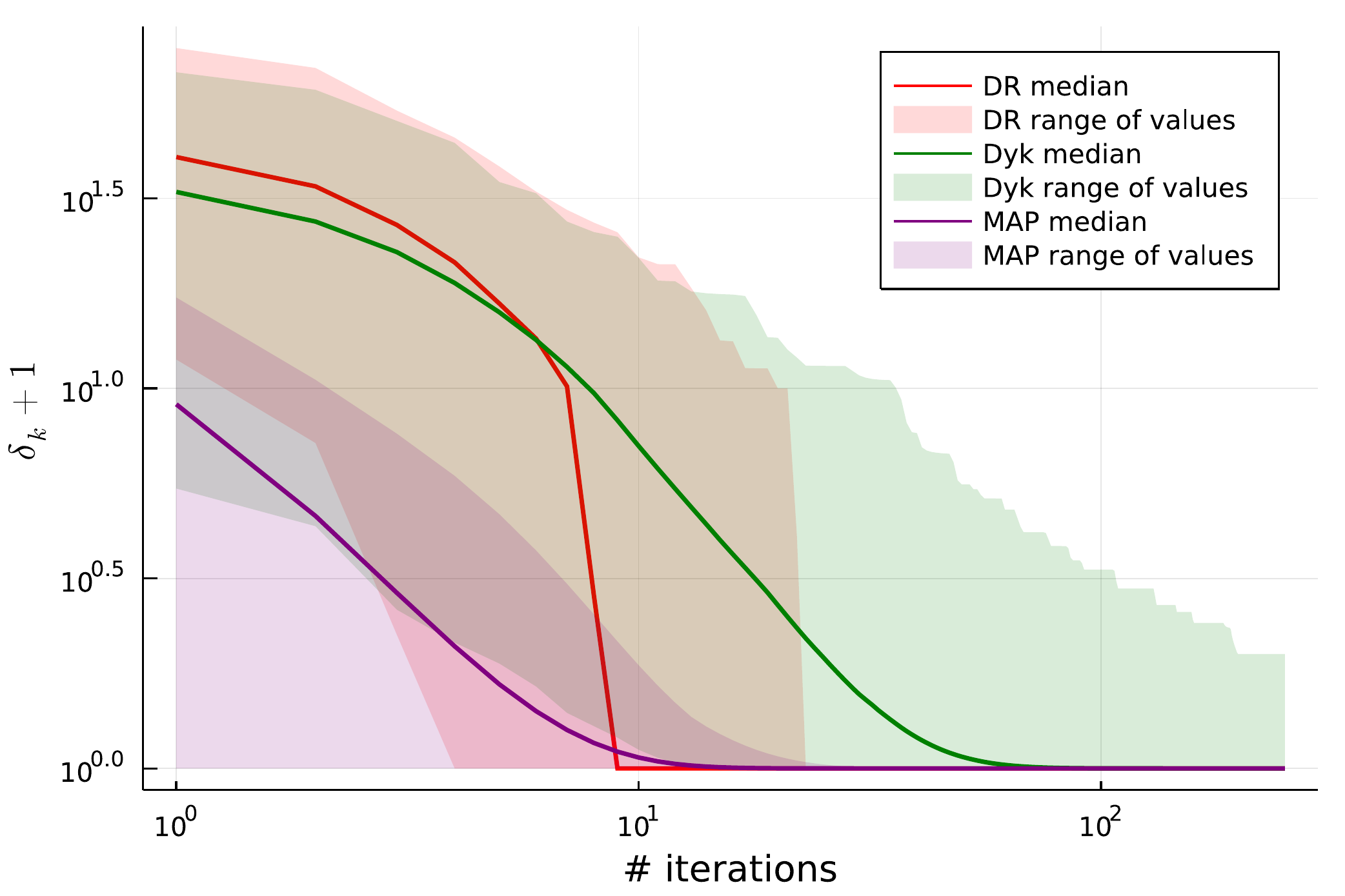}
\caption{Convergence of iterates with the nonnegative matrix constraint}
\label{f:NNcvg}
\end{figure}

Each algorithm is run for 250 iterations and 
for $100,000$
instances of $T_0$ 
that are produced with entries generated uniformly in $[-100,100]$.
The plot of the median value for $\delta_k$ of the iterates is shown in \cref{f:NNcvg}. 
The shaded region for each line represents the range of values attained at that iteration. 
{We assume an algorithm to have achieved feasibility when $\delta_k=0$.}
While MAP and DR always achieve feasibility, 
{as can be seen from the range of their values in \cref{f:NNcvg}},
DR achieves it the fastest in most cases. 
To support this, we order these algorithms in \cref{table:NN} according to their performance. 
The first column reports what percent of the instances achieved feasibility 
in the given order and if any of the algorithms did not converge. 
So the row labelled ``DR$<$MAP'' represents cases where DR achieved feasibility the fastest, 
MAP was second and Dyk did not converge. 
The second column report what percent of the first feasible matrices obtained were closest to 
the starting point  {$T_0$} in the given order. 
 {
This is done by measuring $\norm{T_0-T}$,
where $\norm{\cdot}$ is the operator norm, and $T_k$ is the first feasible matrices obtained
using a given algorithm (Dyk, DR or MAP)
}
We consider the algorithms tied, 
if the distance between the starting point and the estimate for both 
differs by a value less than or equal to $10^{-15}$. 
As is evident, a majority of the cases have DR in the lead for feasibility. 
However, the distance of these matrices is not as close as the ones given by 
MAP and Dyk when feasible. This is consistent with the fact that 
DR explores regions further away from the starting point to look for matrices 
and Dyk is built to achieve the least distance. 
It's also worth noting that at least one of these algorithms converges in every instance. (Convergence for all three algorithms
is guaranteed in theory.)

{ 
Last but not least, because our problem 
deals with \emph{unscaled} row and column sums, we point out that the sought-after projection may also 
be computed by using the algorithm proposed by Calvillo and Romero 
\cite{CaRo} which even converges in finitely many steps!
}

\begin{table}[h]
\centering\renewcommand{\arraystretch}{1.4}
\begin{tabular}{l|r|r}
 \toprule
Outcome  & \begin{tabular}[c]{@{}r@{}}using iterations\\ for feasibility\end{tabular} & \begin{tabular}[c]{@{}r@{}}using distance\\ from $T_0$\end{tabular}  \\
\midrule 
DR$=$MAP   & 3     & -                \\
\rowcolor{rowalt} 
DR$=$MAP$<$Dyk & 21    & -                \\
DR$<$MAP   & 21,205    & -                \\
\rowcolor{rowalt} 
MAP$<$DR   & 2     & 21,210         \\
DR$<$MAP$=$Dyk & 3     & -                \\
\rowcolor{rowalt} 
DR$<$MAP$<$Dyk & 78,708    & -                \\
DR$<$Dyk$<$MAP & 11     & -                \\
\rowcolor{rowalt} 
MAP$<$DR$<$Dyk & 47     & -                \\
Dyk$<$DR$<$MAP & -           & 11          \\
\rowcolor{rowalt} 
Dyk$<$MAP$<$DR & -           & 78,779         \\
\midrule
Total   & 100,000       & 100,000\\
\bottomrule
\end{tabular}
\caption{Results for nonnegative matrices }
\label{table:NN}
\end{table}

\subsection{The nonconvex case} 
\label{subsec:Int}

We exactly repeat the experiment of \cref{ss:NN} with 
the only difference being that the (new) set $A$ in this 
section is the intersection of the 
(old) set $A$ from the previous section (see \cref{e:Aold}) 
and $\mathbb{Z}^{4\times 5}$. This enforces nonnegative 
integer solutions. The projection operator $P_A$ is obtained 
by simply rounding after application of $P_A$ from \cref{ss:NN}.

In this nonconvex case, MAP fails to converge in most cases, 
whereas DR and Dyk converge to solutions as shown in \cref{f:Intcvg}. 
This is corroborated by \cref{table:Int,} where the rows where MAP converges corresponds to 
only a quarter of the total cases.
Again, DR achieves feasibility the fastest in more than half the cases, 
but Dykstra's algorithm gives the  {solution closest to $T_0$ among these,
as shown in the second column of \cref{table:Int}.} 
In this nonconvex case convergence of the any of the algorithms
is not guaranteed; in fact, 
there are several instances when no solution is found. 
However, in the $10^5$ runs considered, 
we did end up discovering several distinct solutions (see \cref{table:Intunique}). 
It turned out that all solutions found were distinct even 
across all three algorithms resulting in 
113622 different nonnegative integer solutions in total.

\begin{figure}[ht]
\centering
\includegraphics[width=0.7\textwidth]{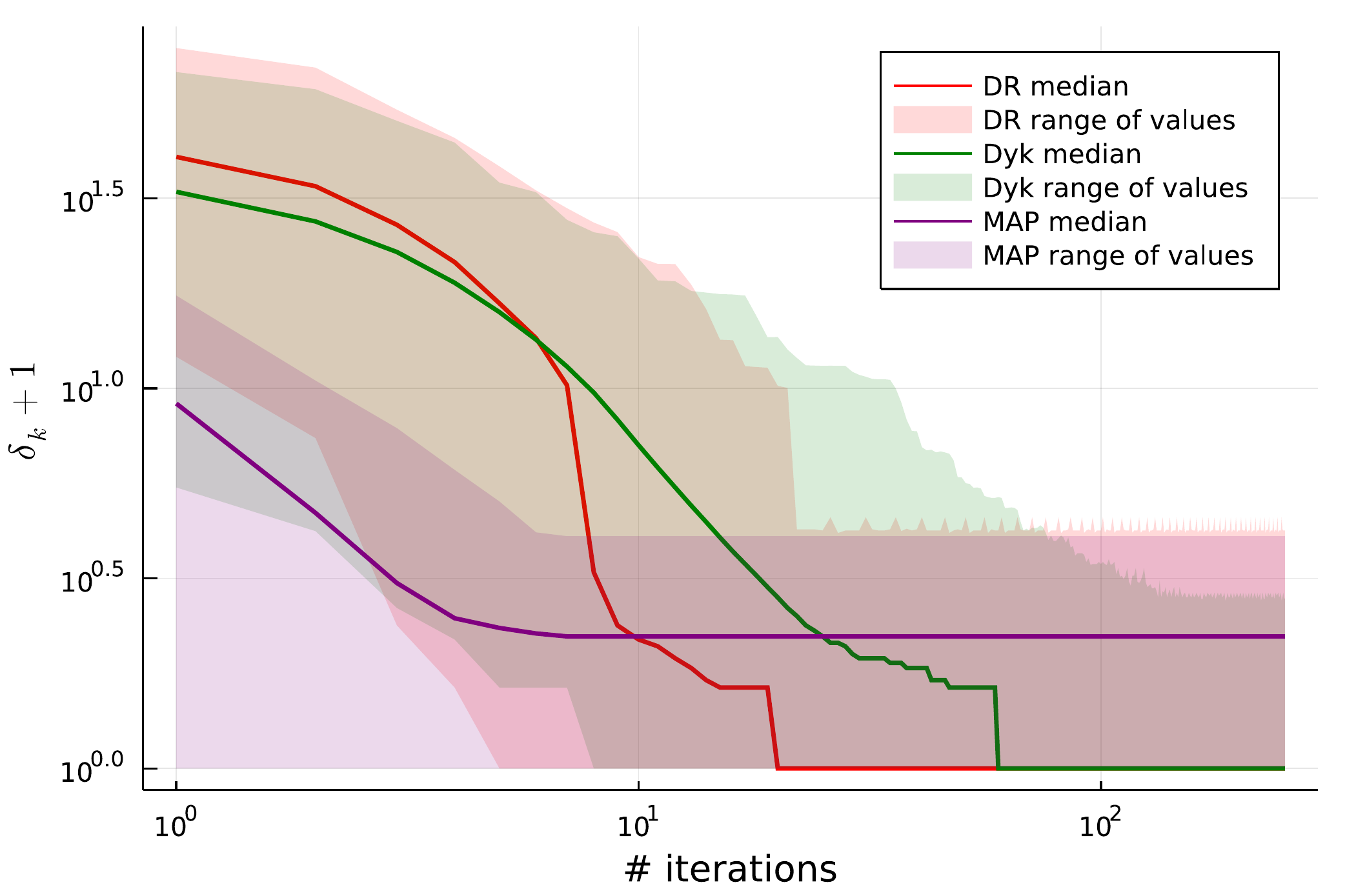}
\caption{Convergence of iterates with the integer matrix constraint}
\label{f:Intcvg}
\end{figure}


\begin{table}[H]
\centering
\centering\renewcommand{\arraystretch}{1.3}
\begin{tabular}{l|r|r}
 \toprule 
Outcome  & \begin{tabular}[c]{@{}r@{}}using iterations\\ for feasibility\end{tabular} & \begin{tabular}[c]{@{}r@{}}using distance\\ from $T_0$\end{tabular}  \\
\midrule
None& 11,694     & 11,694 \\
\rowcolor{rowalt}
DR    & 19,468 & 19,468          \\
MAP   & 25,380   & 25,380               \\
\rowcolor{rowalt}
Dyk   & 45     & 45         \\
DR$=$Dyk & 164     & -                \\
\rowcolor{rowalt}
DR$<$Dyk & 41,822    & -                \\
MAP$<$DR & 73     & 73                \\
\rowcolor{rowalt} 
MAP$<$Dyk & 69     & -                \\
Dyk$<$DR & 1156           & 43,142          \\
\rowcolor{rowalt}
Dyk$<$MAP & -           & 69        \\
MAP$<$Dyk$<$DR & 1     & -                \\
\rowcolor{rowalt} 
MAP$<$DR$<$Dyk & 128     & -                \\
Dyk$<$MAP$<$DR & -          & 129          \\ 
\midrule
Total   & 100,000       & 100,000\\
\bottomrule
\end{tabular}
\caption{Results for nonnegative integer matrices }

\label{table:Int}
\end{table}

\begin{table}[H]
\centering\renewcommand{\arraystretch}{1.4}
\begin{tabular}{l|r|r}
 \toprule
Algorithm    & Solutions found & Unique cases \\
\midrule 
\rowcolor{rowalt} 
DR   & 62,812     & 62,812                \\
MAP &  316   & 314                \\
\rowcolor{rowalt} 
Dyk   & 68,725    & 50,496                \\
\midrule
Total   & 131,853       & 113,622\\
\bottomrule
\end{tabular}
\caption{Integer matrix solutions found by the three algorithms}
\label{table:Intunique}
\end{table}

\small
\noindent
{\bf Declaration}

\footnotesize
\noindent
{\bf Availability of data and materials}
The datasets generated during and/or analysed during the current study are available from the corresponding author on reasonable request.\\
{\bf Competing interests}
The authors declare that they have no competing interests.\\
{\bf Funding}
The research of HHB and XW was partially supported 
by Discovery Grants from the Natural Sciences and Engineering Research Council of Canada.\\
{\bf Authors' contributions}
All authors contributed equally in writing this article. 
All authors read and approved the final manuscript.\\
{\bf Authors' acknowledgments}
{  
We thank the editor Aviv Gibali, three anonymous reviewers, and Matt Tam for 
constructive comments and several pointers to literature we were 
previously unaware of.
}


\small 

\begin{thebibliography}{999}
\sepp

\bibitem{stack}
Aneikei (\url{https://math.stackexchange.com/users/378887/aneikei}), 
Find a matrix with given row and column sums, 2016, 
\url{https://math.stackexchange.com/questions/1969542/find-a-matrix-with-given-row-and-column-sums}


\bibitem{ACT}
F.J.\ Arag\'on Artacho, R.\ Campoy, and M.K.\ Tam,
Strengthened splitting methods for computing resolvents, 
November 2020, \url{https://arxiv.org/abs/2011.01796}. 

\bibitem{BC2017}
H.H.\ Bauschke and P.L.\ Combettes,
{\it Convex Analysis and Monotone Operator Theory
in Hilbert Spaces}, second edition, Springer, 2017.

{ 
\bibitem{CaRo}
G.\ Calvillo and D.\ Romero,
On the closest point to the origin in transportation polytopes,
\emph{Discrete Applied Mathematics}~210 (2016), 88--102.
\url{https://doi.org/10.1016/j.dam.2015.01.027}
}


\bibitem{Gluntetal1}
W.\ Glunt, T.L.\ Hayden, and R.\ Reams,
The nearest ``doubly stochastic'' matrix to a real matrix
with the same first moment, 
{\it Numerical Linear Algebra with Applications}~5 (1998), 475--482. \url{https://doi.org/10.1002/(SICI)1099-1506(199811/12)5:6\%3C475::AID-NLA155\%3E3.0.CO;2-5}


\bibitem{Groetsch}
C.W.\ Groetsch,
{\it Generalized Inverses of Linear Operators},
Marcel Dekker, 1977. 

\bibitem{KadiRing1}
R.V.\ Kadison and J.R.\ Ringrose,
{\it Fundamentals of the Theory of 
Operator Algebras~I: Elementary Theory},
American Mathematical Society, 1997.

\bibitem{Khoury}
R.N.\ Khoury,
Closest matrices in the space of generalized doubly
stochastic matrices,
{\it Journal of Mathematical Analysis and Applications}~222 (1998), 561--568.
\url{https://doi.org/10.1006/jmaa.1998.5970}

\bibitem{Nurminski}
E.A.\ Nurminski,
Single-projection procedure for linear optimization,
{\it Journal of Global Optimization}~66 (2016), 95--110.
\url{https://doi.org/10.1007/s10898-015-0337-9}

{ 
\bibitem{ReedSimon}
M.\ Reed and B.\ Simon,
{\it Methods of Modern Mathematical Physics I: Functional Analysis},
revised and enlarged edition, Academic Press, 1980. 
}

{ 
\bibitem{Romero}
D.\ Romero, 
Easy transportation-like problems on $K$-dimensional arrays,
{\it Journal of Optimization Theory and Applications}~66 (1990), 137--147. 
\url{https://doi.org/10.1007/BF00940537}
}

\bibitem{Takouda05}
P.L.\ Takouda,
Un probl\`eme d'approximation matricielle:
quelle est la matrice bistochastiqu la plus 
proche d'une matrice donn\'ee ?,
{\it RAIRO Operations Research}~39 (2005), 35--54.
\url{https://doi.org/10.1051/ro:2005003}

{ 
\bibitem{wiki:dt}
Wikipedia, 
Discrete tomography,
\url{https://en.wikipedia.org/wiki/Discrete_tomography}, 
retrieved September 13, 2021. 
}
\end{thebibliography}
\end{document}